\documentclass[11pt]{article}
\usepackage{geometry}                
\usepackage{graphicx}
\usepackage{amsmath,amssymb,amsthm}
\usepackage{epstopdf}
\newcommand{\R}{{\mathbb R}}
\newcommand{\id}{{\mathbf 1}}
\newcommand{\rem}[1]{}

\newtheorem{theorem}{Theorem}

\newtheorem{lemma}[theorem]{Lemma}

\newcommand{\p}{{\mathbf p}}
\newcommand{\q}{{\mathbf q}}
\newcommand{\mtot}{{\cal M}}

\renewcommand{\v}{{\mathbf v}}

\DeclareMathOperator*{\diag}{diag}
\renewcommand{\sp}{\mathfrak{sp}}
\newcommand{\so}{\mathfrak{so}}

\newcommand{\trans}{\mathbf q_0}
\newcommand{\velshift}{\mathbf v_0}

\newcommand{\nfootnote}[1]{}

\title{The Lie-Poisson structure of the reduced $n$-body problem}
\author{
      H.~R. Dullin   \thanks{Supported in part by ARC grant DP110102001} \\
     School of Mathematics and Statistics\\
     The University of Sydney\\
     Sydney, NSW 2006, Australia\\
     {\tt Holger.Dullin@sydney.edu.au} 
}
\date{\today}                                           
\begin{document}
\maketitle
\begin{abstract}
\vspace*{1ex}
\noindent
The classical $n$-body problem in $d$-dimensional space is invariant under the Galilean symmetry group.
We reduce by this  symmetry group using the method of polynomial invariants. 
As a result we obtain a reduced system with a Lie-Poisson structure which is isomorphic 
to $\sp(2n-2)$, independently of $d$. The reduction preserves the natural form of 
the Hamiltonian as a sum of kinetic energy that depends on velocities only and 
a potential that depends on positions only. Hence we proceed to construct 
a Poisson integrator for the reduced $n$-body problem using a splitting method.
\\[1ex]
\end{abstract}


\section{Introduction}

The $n$-body Hamiltonian in $\R^d$ is 
\begin{equation} \label{eqn:Ham}
   H = \sum_{i=1}^n \frac{||\p_i||^2}{2 m_i} + \sum_{1 \le i < j \le n} V_{ij}( ||\q_i - \q_j||^2) \,,
\end{equation}
with position vectors $\q_i \in \R^d$ and conjugate momentum vectors $\p_i \in \R^d$, $i = 1, \dots, n$.
All of the following is valid for a more general potential $V$ that is a function of pairwise distances $ ||\q_i - \q_j||^2$ only.
In the gravitational case the pair-wise potential function is $V_{ij}(r^2) = -G m_i m_j / r$. 
The equations of motion are
\begin{equation} \label{eqn:ode}
   \dot \q_i = \frac{1}{m_i}  \p_i, \quad \dot \p_i = - 2 \sum_{j \not = i} (\q_i - \q_j)V'_{ij}( ||\q_i - \q_j||^2)  \,.
\end{equation}

We are going to follow Lagrange's footsteps \cite{Lagrange1772}, as was more recently done by Albouy 
and Chenciner \cite{AlbouyChenciner98,Albouy04},
by introducing scalar products of difference vectors as new coordinates.
In fact, the Poisson structure we are going to derive is mentioned in \cite{AlbouyChenciner98}, 
but they were mostly interested in relative equilibria, and hence only mentioned the Poisson structure 
in passing in the final paragraph of their first section. Here we are going to derive explicit expressions for the 
Poisson structure in a particularly suitable basis and then use the Poisson structure to construct a 
geometric integrator for the reduced 3-body problem.

A related approach has been taken in \cite{BolsinovBorisovMamaev99}, but there mainly the 
case of vanishing angular momentum for three bodies is considered.
The Poisson reduction of  rotational $O(d)$ symmetry  has been discussed for arbitrary $n$ in  \cite{LMS93}
from the point of view of singular reduction.
In many ways our analysis is similar to \cite{LMS93}. 
The main difference is that they only reduced $O(d)$ symmetry, but not the Euclidean or Gallilean symmetry,
and that in parts they only considered angular momentum zero.

The general idea of Poisson reduction using invariants is as follows.
Say we have a symmetry given as a group action $\phi^g$ of the group $G$ on a Poisson manifold $M$.
Then by definition an invariant $f : M \mapsto \R$ satisfies $f \circ \phi^g = f$ for all $g \in G$.
If for every  $g$ the mapping $\phi^g$ preserves the Poisson bracket, 
then in addition we have that for two invariants $f$ and $g$ their Poisson bracket
$\{ f, g \} = \{ f \circ \phi^g, g \circ \phi^g \} = \{ f, g \} \circ \phi^g$ is also invariant under $\phi^g$. 
Reduction may be possible for a set of invariants
whose Poisson brackets with each other are closed, so that we can try to define 
a reduced Poisson bracket which has the invariants as new coordinates.
In general this reduced bracket may not satisfy the Jacobi identity in the whole reduced space,
but only on some subset defined by the syzygies between the invariants, see e.g.~\cite{Egilsson95},
and thus it may not give a Poisson bracket on the space of invariants.
However, if the bracket is linear in the invariants, the so called Lie-Poisson case, 
then the Jacobi identity is inherited from the Lie-group structure of the 
brackets of the invariants.

As pointed out in \cite{LMS93} a setting where Poisson reduction using invariants automatically reduces to the Lie-Poisson case 
is one where the invariants are quadratic forms in the original Euclidean variables for which the original Poisson bracket 
has a constant structure matrix (e.g.~the standard symplectic matrix). Since the bracket of 
two quadratic forms then again is a quadratic form the closeness then implies that the reduced bracket
is Lie-Poisson. As we will see this is the case that occurs for the Galilean symmetry of the $n$-body problem. 

Poisson reduction using invariants differs significantly from symplectic reduction, as initiated 
by Marsden and Weinstein~\cite{MW74}. Specifically 
a good review of symplectic reduction for the $n$-body problem  is given in \cite{LittlejohnReinsch97}. Translation symmetry is often reduced 
by introduction of so-called Jacobi coordinates relative to the centre of mass. 
Reduction by the rotational symmetry is more difficult because the rotation does not act freely.
A variety of coordinate systems have been employed for this,
again see \cite{LittlejohnReinsch97} for a discussion.
The classical approach is the so called elimination of the nodes, see e.g.~\cite{Whittaker37}.
This gets down to 4 degrees of freedom, but creates a singularity 
for phase space points in which all position and momentum differences are collinear so that the rotational symmetry does not act freely.
Poisson reduction using invariants instead is well suited to handle the resulting singular reduction.

To my knowledge, no matter what coordinates are used, the fully symplectically reduced 
Hamiltonians have a kinetic energy that depends on the positions. 
Using invariants for the reduction keeps the positions and momenta separate in 
the Hamiltonian and  treats all bodies equally.
In particular this will enable us to construct a numerical integration schemes that preserve the Poisson structure
using a splitting method \cite{McLaQu02,HaLuWa02,LeimkuhlerReich04}.
\nfootnote{Other symplectic integration methods for the $n$-body problem have been used \cite{Wisdom,Laskar?}.
Their emphasis is ...}


\section{Galilean symmetry of the $n$-body problem}

\nfootnote{add a comment explaining why it is not good enough to reduce by $SE(d)$, or is it? Like this?:
The momenta themselves are invariant under $SE(d)$, so it is possible to consider 
scalar product of all momentum vectors and all position difference vectors?
Then we can write $H$ in terms of those, but are they close under PB?}
The Hamiltonian \eqref{eqn:Ham} is invariant under 
translations in space $T(\q_i, \p_i) = (\q_i + \trans, \p_i)$, $i=1,\dots, n$ with $\trans \in \R^d$
and rotations $R(\q_i, \p_i) = ( R \q_i , R \p_i)$, $i=1,\dots, n$ with $R \in O(d)$.
In addition Galilean boosts $B (\q_i, \p_i) = (\q_i + \velshift  t, \p_i + m_i \velshift )$, $\velshift  \in \R^d$ 
leave the equations of motion \eqref{eqn:ode} invariant but {\em not} the Hamiltonian \eqref{eqn:Ham}.
The total symmetry group is the Galilean group $G(d)$ where in our case we include the reflections $O(d)/SO(d)$.
If one replaces the $O(d)$ subgroup by $SO(d)$ the set of invariants becomes more complicated for $d>2$, 
as pointed out by \cite{LMS93,LittlejohnReinsch97}, see below.
The Galilean group also contains a generator of time translations, and invariance of the Hamiltonian under 
this transformation leads to conservation of energy. This group element does not play a role in the following,
since we do not want to reduce by this symmetry.
\nfootnote{well, so is it reduction by $G(d)$ or not?}

Translation invariance of the Hamiltonian by Noether's theorem leads to the conservation of linear momentum ${\bf P} = \sum \p_i$.
Rotational symmetry of the Hamiltonian leads to the conservation of angular momentum with respect to the origin
$
    {\bf L} = \sum \q_i \wedge \p_i  \,.
$
In general the two-form ${\bf L}$ may be viewed as an anti-symmetric matrix,
but for the special case $d=3$ it can also be identified with a vector in $\R^3$.

The behaviour of the integrals  $\bf P$, $\bf C$, $\bf L$, $H$ under 
translation by $\trans$, boost by $\velshift$, rotations by $R$ and time translations by $\tau$ is given in the following table.
We always assume that $\mtot \not = 0$.
\[
\begin{array}{c|ccccc}
    &  {\bf P} & {\bf C} & {\bf L} & H \\ \hline
T & {\bf P} & {\bf C} + \trans & {\bf L} + \trans \wedge {\bf P} & H \\
B & {\bf P}+  \velshift  {\cal M} & {\bf C}+ \velshift  t & {\bf L}+  (\mtot  {\bf C} - t {\bf P}) \wedge \velshift  & H + {\bf P} \cdot \velshift  + \frac12 {\cal M} || \velshift  ||^2 \\
R & R {\bf P} & R{\bf C} & R {\bf L} & H \\
H & {\bf P} & {\bf C} + \tau {\bf P}/\mtot & {\bf L} & H
\end{array}
\]
The last row shows that ${\bf P}$ and ${\bf L}$ are constants of motion, as well as $\mtot {\bf C} - t {\bf P}$ (the generator of boosts), 
which implies that the centre of mass moves along a straight line with constant velocity  $\dot {\bf C} = {\bf P}/\mtot$.
The vector integrals are not invariant under rotations but instead equivariant under rotations, as shown in the third row.
None of the integrals is invariant under boosts, the additive changes are indicated in the second row. 
Note, however, that the combination $H_c = H - \frac{1}{2\mtot} ||{\bf P}||^2$ is invariant under 
boosts. This Hamiltonian is the energy up to the kinetic energy in the overall translation, and $H_c$ is the
Hamiltonian that will be reduced later on.
Similarly the combination ${\bf L}_c = {\bf L} - {\bf C} \wedge {\bf P}$ is invariant under translations and boosts. 
This is the angular momentum with respect to the centre of mass, see below.

Reduction by translation symmetry removes $d$ degrees of freedom; 
the translation reduced system has $nd - d$ degrees of freedom. 
Reduction by rotational symmetry in addition removes $d-1$ degrees of freedom,
since there are only $d-1$ commuting integrals when ${\bf L} \not = 0$ 
that can be constructed from the $d(d-1)/2$ entries of the 2-form ${\bf L}$.
When ${\bf L} = 0$ further reduction is possible.
For general ${\bf L}$ the fully reduced phase space has dimension $nd - 2d + 1$.

When the action of the symmetry is not free singularities may arise in the reduced phase space. 
When the position and momentum vectors in the 3-body problem are collinear then all rotations
that have this line as rotation axis leave the point in phase space fixed, so that there is 
non-trivial isotropy, and hence singular reduction. 
Singularities arise similarly at points in phase space for which position and momentum vectors
span a subspace whose dimension is smaller than $d$ and there are non-trivial rotations in $\R^d$ that fix this subspace. 
Thus for $d=3$ planar motions are not singular, but for $d > 3$ planar motions lead to singular reduction. 
Similarly for $d=2$ collinear motions are not singular.

\subsection*{Centre of mass decomposition}

Reduction by translation symmetry can be achieved by introducing 
a coordinate system on configuration space in which the centre of mass ${\bf C}$ is a new coordinate 
together with $n-1$ relative position vectors. 
A popular choice for such coordinates are the Jacobi coordinates,
since they keep the kinetic energy diagonal,
see e.g. \cite{LittlejohnReinsch97} and the references therein.
One disadvantage of Jacobi coordinates is that they do not 
treat the bodies equally, which is particularly annoying when 
dealing with the case of equal masses. Our approach 
preserves this discrete symmetry. 

As we have seen the Hamiltonian is not invariant under boosts. 
Clearly the potential energy is invariant under the full symmetry group, 
hence the kinetic energy is not invariant under boosts. To remedy this 
we split the kinetic energy into the kinetic energy of the motion of the 
centre of mass and the remainder denoted by $K_c$.
This relative kinetic energy $K_c$ is invariant under the full symmetry group.
$K_c$ can be found by replacing $\q_i$ by $\q_i - {\bf C}$ and $\p_i$ by $\p_i - \frac{m_i}{\mtot}{\bf P}$
 (or equivalently by replacing $\dot \q_i$ by $\dot \q_i - \dot {\bf C}$) in the original kinetic energy, 
i.e.\ by measuring positions and velocities relative to the centre of mass.
Thus we find the {\em kinetic energy relative to the centre of mass} as
\begin{equation} \label{eqn:Kc}
    K_c = \frac12 \sum_{i=1}^n \frac{1}{m_i} || \p_i - \frac{m_i}{\mtot} {\bf P}||^2  =  \frac12 \sum_{i=1}^n  \frac{1}{m_i}  || \p_i||^2  - \frac{1}{2 \mtot} || {\bf P}||^2
\end{equation}
so that $H_c = K_c + V = H -  \frac{1}{2 \mtot} || {\bf P}||^2$ is invariant under translations and boosts.
The original Hamiltonian hence is $H = H_c + \frac{1}{2 \mtot} || {\bf P}||^2$.
The invariance of $K_c$ is obvious because both
 $ \p_i$ and $\frac{m_i}{\mtot} {\bf P}$ are changed by $m_i \velshift $ 
under boosts $B$.
Invariance of $K_c$ can also be verified using the table showing the effect of the $G(d)$ group action on the constants of motion.
In fact $H_c$ is the Casimir of $G(d)$ as already mentioned.

Measuring the total momentum relative to the centre of mass only gives zero: 
${\bf P}_c = \sum (  \p_i - \frac{m_i}{\mtot} {\bf P}) = 0$
and similarly for the centre of mass  ${\bf C}_c = \sum m_i ( \q_i - {\bf C}) = 0$.
Finally we also define the {\em angular momentum about the centre of mass} as
\begin{equation} \label{eqn:Lc}
    {\bf L}_c = \sum ( \q_i - {\bf C}) \wedge ( \p_i - \frac{m_i}{\mtot} {\bf P}) 
    = \sum \q_i \wedge \p_i -  {\bf C} \wedge {\bf P}
    \,.
\end{equation}
Invariance of ${\bf L}_c$ under boosts $B$ is now obvious, as it was for $K_c$.
The length squared of the angular momentum relative to the centre of mass $||{\bf L}_c||^2$ 
is hence invariant under the full $G(d)$ action. 
What we have done here is the well known derivation of the Casimirs $K_c$ and $||{\bf L}_c ||^2$ of the 
Galilean group, see e.g.~\cite{SudarshanMukunda74}. 
This prepares the system for reduction using quadratic polynomial invariants.

\section{Reduction using polynomial invariants}

The Hilbert-Weyl theorem (see, e.g., \cite{GSS88})
guarantees the existence of a so called Hilbert basis, i.e.\ a finite set of polynomials
that generate the ring of invariant polynomials for a compact group acting linearly on a vector space. 
By a theorem of G.~Schwarz \cite{Schwarz75} even every smooth invariant function can be expressed as 
a smooth function of the basic polynomial invariants. E.g.\ the invariant functions $K_c$, $V$,
and $|| {\bf L}_c ||$ can all be written in terms of the basic polynomial invariants.
The Galilean group $G(d)$ is not compact, and the $G(d)$ action is affine instead of linear, 
so strictly speaking the theorem does not apply in our case.
Nevertheless, we will see that the invariants of the $G(d)$ action 
are quadratic functions in the original variables and that their Poisson bracket is closed,
i.e.\ every bracket of invariants can again be expressed in terms of invariants.

The invariants are introduced in two simple steps. 
First form difference vectors 
\[
   \q_{ij} = \q_i - \q_j \quad \text{and}\quad \v_{ij} = \dot \q_i - \dot \q_j = \p_i/m_i  - \p_j/m_j
\]
which are invariant under translations and boosts. 
Notice that momentum differences are not invariant under boosts unless all masses are equal. 
Second take scalar product of these difference vectors that are invariant under rotations as well.
Hence scalar products of these difference vectors are invariant under $G(d)$.

Since difference vectors are linear in the original coordinates, 
these invariants are quadratic in the original coordinates. 
If we would consider $SO(d)$ instead
of $O(d)$ reduction for $d > 2$ there would be additional invariants given by
the signs of the determinants of any $d$ difference vectors \cite{LMS93}. 
The square of such a determinant, however, can again be expressed in terms of scalar products:
it is the Gram determinant of the $d$ vectors.
\footnote{Even though there are additional invariants so that the quadratic invariants do not form 
a Hilbert basis for the invariants, the quadratic invariants alone still have a closed Poisson bracket.
So for $SO(d)$ reduction the sub-algebra of the quadratic invariants is sufficient for reduction, 
even though it is not a Hilbert basis.}

There are $n(n-1)/2$ non-zero difference vectors between $\q_i$, similarly for $\v_j$, 
but only $n-1$ of each group are independent. 
A possible choice of basis difference vectors are the $2n-2$ vectors
$\q_{1j}, \v_{1j}$, $j=2, \dots, n$, similarly for any other fixed first (or second) index.
The elements of this vector space are invariant under translations and boosts.

Now we have reduced the non-compact part of $G(d)$ by introducing difference vectors.
In the next step we introduce invariants based on these differences vectors that are $SO(d)$ invariant:
Any scalar product between two vectors from the space of difference vectors is invariant under $SO(d)$, and 
hence invariant under the full symmetry group.
The basis of the space of difference vectors has dimension $2n - 2$, 
and forming all pairs there are $(2n-2)(2n-1)/2$ fully invariant scalar products.
We will later formally show that this number is the dimension of the vector space of quadratic invariants, 
which is isomorphic to the Lie algebra $\sp(2n-2)$, see Theorem~\ref{thm:spfin}.

All the scalar products between the basic difference vectors can be conveniently 
combined in a Gram matrix.
Instead of taking the entries of the  Gram matrix as invariants we will choose certain 
linear combinations that are more natural because they appear in the reduced Hamiltonian.
In particular the potential depends on the $n(n-1)/2$ mutual distances $\rho_{ij} = ||\q_{ij}||^2$, $i < j \le n$,
\[
    V = \sum_{1 \le i < j \le n} V_{ij}( \rho_{ij} )  = - G \sum_{1 \le i < j \le n} m_i m_j \frac{1}{\sqrt{\rho_{ij} }} \,,
\]
where the second equality holds for the gravitational $n$-body problem.
The full kinetic energy of the Hamiltonian \eqref{eqn:Ham} can {\em not} be written in terms 
of the invariants because it is not invariant under boosts. 
However, the kinetic energy {\em relative to the centre of mass} $K_c$ is invariant and can be rewritten in terms of the $n(n-1)/2$
relative speeds  $\nu_{ij} = ||\v_{ij}||^2$, $i < j \le n$ as
\[
   K_c =\frac{1}{2\mtot} \left( \mtot \sum m_i ||\dot \q_i||^2 - \left( \sum m_i \dot \q_i \right)^2 \right) =   \frac{1}{2\mtot} \sum_{1 \le i < j \le n} m_i m_j \nu_{ij} \,.
\]
The terms of the form $m_i^2 \dot \q_i^2$ cancel since they appear in both sums, and the remaining $(n-1)n$ terms from the first sum
recombine with the $n(n-1)/2$ of the second sum to give the result.

Similarly the angular momentum relative to the centre of mass can be written in terms of difference vectors
using the identity $\mtot(\q_i - {\bf C}) = \sum_j m_j \q_{ij}$ and its time derivative, so that
\[
  {\bf L}_c = \frac{1}{\mtot^2} \sum_{i,j,k} m_i m_j m_k \q_{ij} \wedge \v_{ik} \,.
\]
To construct an invariant from the equivariant ${\bf L}_c$ we compute the length squared of this vector.
The identity $(a \wedge b) \cdot (c \wedge d) = (a \cdot c)( b \cdot d) -( a \cdot d)( b \cdot c)$
then shows that $||{\bf L}_c||^2$ can be written in terms of invariants.
\footnote{For general dimension $d>3$ where $L$ is an antisymmetric rank 2 matrix ($a \wedge b = a b^t - b a^t$) 
the scalar product  between two such matrices $A$ and $B$ is given by ${\mathrm tr}( A B^t)/2$, 
and the Lagrange identity still holds.}
In this expansion of total angular momentum in terms of invariants not only 
$\rho_{ij}$ and $\nu_{ij}$ appear, 
but also scalar products between $\q_{ij}$ and $\v_{kl}$.
Exactly which scalar products to choose as a basis of the vector space of quadratic invariants we leave open 
until section~\ref{sec:BasisPoi} where the structure matrix of the reduced Poisson bracket is computed in a certain basis. 
Until then we work in a basis-independent formulation.

The moment of inertia with respect to the centre of mass can be 
similarly expressed as 
\[
 I_c = \sum m_i (\q_i - \mathbf{C})^2 = \frac{1}{\mtot} \sum m_i m_j \rho_{ij} \,.
\]

We will show (see Theorem~\ref{thm:spfin}) that the dimension of the space of quadratic invariants is $(2n-1)(n-1)$,
independent of $d$. This number is given in the column denoted by $\mathrm{Inv}$ in the following tables.
In the subsequent columns of the first table the dimension $2(n-1)d$ of the translation
reduced $n$-body problem in dimension $d$ is given for $d = 2,3,4$.
\[
\begin{array}{l|l|lllll}
\multicolumn{6}{c}{\text{translation reduced}} \\
&&\multicolumn{4}{c}{2(n-1)d} \\ \hline
n \setminus d 	& \mathrm{Inv}  & 1 & 2 & 3 & 4 \\ \hline
2      			& 3  & 2 & 4 & 6 & 8  \\
3      			& 10 & 4 & 8 & 12 & 16 \\
4      			& 21 & 6 & 12 & 18 & 24 \\
5      			& 36 & 8 & 16 & 24 & 32 \\
\end{array}
\qquad \qquad
\begin{array}{l|l|lllll}
\multicolumn{6}{c}{\text{fully reduced}} \\
&&\multicolumn{4}{c}{2(n-1)d - 2(d-1)} \\ \hline
n \setminus d 	& \mathrm{Inv}  & 1 & 2 & 3 & 4 \\ \hline
2      			& 3  & 2 & 2 & 2 & 2  \\
3      			& 10 & 4 & 6 & 8 & 10 \\
4      			& 21 & 6 & 10 & 14 & 18 \\
5      			& 36 & 8 & 14 & 20 & 26 \\
\end{array}
\]
In the second table the dimensions of the fully reduced $n$-body problem are given.
It should be noted, however, that symplectic reduction in this case leads to 
1) a singular reduced space and 
2) equations which cannot be integrated using splitting methods. 
The full reduction by $O(d)$ removes another $2(d-1)$
dimensions (since there are only $d-1$ commuting integrals, even though the number
of integrals grows quadratically with $d$).
\footnote{The classical statement is that 
in the $3$-body problem in dimension $d=3$ there are 10 independent first integrals. 
From the point of view of reduction this counting is unusual for two reasons:
1) the Hamiltonian is counted, 
and 2)  the three angular momenta are counted even though they do not commute, and hence cannot be used to lower the number of degrees of freedom by $d$. 
}
Whichever counting is used,  the tables show that the method
presented here is efficient only when $n=3$ and $d\ge 3$, unless unusually high dimensions $d > 3$ are
considered.
The number of $G(d)$ invariants grows quadratically with $n$, and hence there must be more and more Casimirs
relating these invariants.
The principal strengths of the current approach are that 1) Poisson reduction using invariants
has no problem with singular reduction, 2) it is naturally independent of the dimension $d$,
and 3) it preserves the simple form of the Hamiltonian $H_c$ which allows a splitting integrator to be constructed.

Let us finally comment on the dimension of the so called shape space and its relation 
to the dimension of the fully reduced phase space.
For $d=3$ the dimension of shape space is $3n-6$, namely the dimension of the configurations space of $n$ 
points in $\R^3$ up to translations and rotations. 
This gives dimension 3 for $n=3$, namely the number of sides of a triangle. 
But the reduced system has 4 degrees of freedom, dimension 8, as listed in the the second table.
The additional degree of freedom describes the (fixed) angular momentum vector
in a body frame attached to the plane of the triangle, i.e.\ it describes the position of a point on a sphere,
which is the co-adjoint orbit of $\so^*(3)$.
The rotation of this plane about the axis of the angular momentum vector is 
obtained by reconstruction.

\subsection*{Block form of the total Poisson structure}

The canonical variables $\q_i, \p_i$ satisfy the Poisson
bracket $\{ (\q_i)_k, (\p_j)_l \} = \delta_{ij} \delta_{kl}$, $i,j = 1, \dots, n$, and $k,l = 1, \dots, d$.
A second index outside parenthesis denotes the component of a vector,
not to be confused with the double index without parenthesis $\q_{ij}$ for the difference between $\q_i$ and $\q_j$.
\nfootnote{or use greek and latin lower and upper indices?}

Since the velocities are more important than the momenta in our construction
it is useful to pass to a new non-standard symplectic structure in which the momenta are replaced 
by velocities as variables. This is a non-canonical transformation that changes the symplectic structure.
It introduces scalar factors originating from $\{ (\q_i)_k, \dot (\q_j)_k \} = \{ (\q_i)_k, (\p_j)_k \} /m_j$. 
Thus the usual identity block-matrices in the standard symplectic structure are replaced by diagonal matrices with 
entries $1/m_j$.

To verify that the equations of motion separate into the centre of mass motion and the non-trivial part
described by invariants we need to show that the 
Poisson bracket between any quadratic invariant and the centre of mass 
and its derivative (the linear momentum) vanishes.
Clearly 
$\{ \q_{ij} \cdot \q_{kl} , { \bf C} \}  = 0$ and  
$\{ \v_{ij} \cdot \v_{kl} , { \bf P} \}  = 0$.
Now verify that
$\{ \q_{ij} \cdot \v_{kl} , { \bf C} \}  =  \partial_{v_k}  \q_{ij} \cdot \v_{kl} / \mtot +   \partial_{v_l}  \q_{ij} \cdot \v_{kl} / \mtot = 0$, 
similarly $\{ \q_{ij} \cdot \v_{kl} , { \bf P} \}  = 0$.
Finally 
$\{ \q_{ij} \cdot \q_{kl} , { \bf P} \}  = \sum_{m=i,j,k,l} \partial q_m \q_{ij} \cdot \q_{kl}  = 0$,
and similarly $\{ \v_{ij} \cdot \v_{kl} , { \bf C} \}  = 0$.

Thus when we write down the equations of motion generated by $H = H_c + \frac{1}{2\mtot} ||{\bf P}||^2$
we recover the trivial equation $\dot {\bf C} = \{ H, {\bf C} \} =  {\bf P}/ \mtot$ and 
$\dot {\bf P} = \{ H, {\bf P} \} = 0$. More importantly the equations of motion for the invariants
are determined by $H_c$ only. E.g.\ for $I =  \q_{ij} \cdot \v_{kl}$ the equation of motion is
$\dot I = \{ H, I \} = \{ H_c, I \}$, and similarly for all other invariants. Thus from now on 
$H_c$ is our Hamiltonian, and the main task is to 
compute the Poisson bracket between invariants and express them in terms of invariants.

\subsection*{Two bodies}

Before we proceed to the general case we briefly treat $n=2$ where everything can be done by direct computation.
As a basis for the three invariants we choose $\rho_{12} = ||\q_{12}||^2$, $\nu_{12} = ||\v_{12}||^2$, and $\sigma_{12} = \q_{12} \cdot \v_{12}$.
The relative kinetic energy is $K_c = \mu \nu_{12}$, where the reduced mass $\mu$ is given by  $1/\mu = 1/m_1 + 1/m_2$.
The relative momentum is $||{\bf L}_c||^2 = \mu^2 (\rho_{12} \nu_{12} - \sigma_{12}^2)$.
The reduced relative Hamiltonian is 
\[
    H_c = \frac12 \mu \nu_{12} - \mu \frac{G \mtot }{\sqrt{ \rho_{12} } }\,.
\]
The Poisson brackets between the new variables $Z = (\rho_{12}, \nu_{12}, \sigma_{12})$ can be found by 
direct computation, e.g.\
\begin{align*}
\{ ||\q_1 - \q_2||^2, ||\v_1 - \v_2||^2 \} 
& = \sum 4 ( (\q_1)_i - (\q_2)_i) \{ (\q_1)_i - (\q_2)_i, (\v_1)_i - (\v_2)_i \} ((\v_1)_i - (\v_2)_i ) \\
& = \sum 4  ( (\q_1)_i - (\q_2)_i) \left( \frac{1}{m_1} + \frac{1}{m_2} \right)  ((\v_1)_i - (\v_2)_i )\\
& = \frac{4}{ \mu } (\q_1 - \q_2) \cdot (\v_1 - \v_2) 
\end{align*}
where we have used a 2nd index outside parenthesis to denote the components of a vector.
Notice that the bracket is independent of the spatial dimension $d$.
Thus the non-vanishing brackets are
\[
  \{ \rho_{12}, \nu_{12} \} = \frac{4}{\mu}  \sigma_{12} , \quad
  \{ \rho_{12}, \sigma_{12} \} =\frac{2}{\mu}  \rho_{12}, \quad
  \{ \nu_{12}, \sigma_{12} \} = -\frac{2}{\mu}  \nu_{12} \,,
\]
so that the bracket between the invariants is closed, i.e.\ all brackets of invariants can be 
expressed in terms of invariants.
The structure matrix $B$ of the new Poisson bracket is 
\[
  B = \frac{2}{\mu} \begin{pmatrix}
      0 & 2 \sigma_{12} & \rho_{12} \\
      -2 \sigma_{12} & 0 & - \nu_{12} \\
      -\rho_{12} & \nu_{12} & 0 
   \end{pmatrix} \,.
\]
It is easy to check that this Poisson bracket satisfies the Jacobi identity,
but this is also automatic since it is linear, i.e.{} it is a Lie-Poisson bracket.
In fact it is the Lie-Poisson bracket of $\sp(2)$, see e.g.~\cite{MarsdenRat94}.
As expected $||{\bf L}_c||^2$ is a Casimir of the bracket. 
This Casimir can also be interpreted as the Gram determinant of the matrix that has the 
difference vectors $(\q_{12}, \v_{12})$ as columns.
The generic symplectic leaf $\{ ||{\bf L}_c||^2 = l^2 \}$ for $l^2  > 0$ is one sheet (since $\rho_{12} > 0$, $\nu_{12} > 0$) of a 
two-sheeted hyperboloid. The surface define by $\{ H = h \} $ in the reduced space
is a cylinder in the $\sigma_{12}$ direction over a hyperbola in $(\sqrt{\rho_{12}}, \nu_{12})$.
For large negative $h$ there is no intersection between the two surfaces. 
The first tangency corresponds to the 
circular orbit. The intersection is topologically a circle until $h$ reaches zero. For positive $h$
the intersection is an unbounded line corresponding to a scattering solution of the
Kepler problem. 
When the Casimir is zero the symplectic leaf is a cone so that the reduced phase space
is singular. 
We now describe this case.
The vertex of the cone is not accessible for finite energy. 
The collision corresponds to the infinite point on the positive $\nu$ axis.
All intersection curves asymptote towards this point. 
For positive energy there are two unbounded lines meeting in the collision.
For negative energy the intersections form a tear drop with a turning point 
at finite $\rho$.

\section{Algebra of invariant quadratic forms}

We now derive the reduced Poisson bracket between the invariants 
for general $n$ and without choosing a particular basis of invariants.
Each invariant is quadratic in the original variables. The standard symplectic Poisson bracket of two
quadratic functions yields a quadratic function in the same variables. 
The set of all quadratic functions thus is a sub-algebra of all functions.
The main observation is that the Poisson bracket of the invariants
is closed, i.e.\ brackets of invariants can be written in terms of the invariants again. 
This means there is a smaller sub-algebra of invariant quadratic functions within the 
algebra of quadratic functions. From an abstract point of view this follows as 
soon as it has been established that the quadratic invariants contain a Hilbert basis.
Instead of assuming this in the following we directly show that the quadratic invariants
form a sub-algebra.

Verifying this reduces to linear algebra: Each invariant is a certain quadratic form and together they form a vector space.
The set of all quadratic forms (which can be identified with the space of all symmetric matrices) becomes an 
algebra under the Poisson bracket. 
Because the invariants are quadratic the resulting Poisson structure
will in fact be a Lie-Poisson structure, where the entries of the Poisson structure matrix
are linear in the invariants and the linear combinations are given by the structure constants
of some Lie algebra. 

The next Lemma is well known, see, e.g., \cite{LMS93}, and we will adapt it for our particular case in the following.

\begin{lemma} \label{lem:SP}
The symplectic Poisson bracket on $\R^{2m}$ induces an algebra on quadratic forms that is isomorphic to $\sp(2m)$
\end{lemma}

\begin{proof}
Denote by $Z$ the vector of $2m$ symplectic variables corresponding to (possibly non-standard) symplectic $2m\times 2m$ matrix $J$, where $m$ is an arbitrary positive integer, which in our case will be $m = nd$.
The Poisson bracket is defined by $\{ f, g \} = ( \nabla_Z f,  J  \nabla_Z g)$ where $(,)$ is 
the Euclidean standard scalar product. 
Let $Q_A$ be a quadratic form in Z such that the Hessian of $Q_A$ is the symmetric $2m \times 2m$ matrix $A$, namely 
$Q_A = \frac12 (Z, A Z)$.
The Poisson bracket of two quadratic form induces an algebra of symmetric $2m\times 2m$ matrices defined by
\[
    \{ Q_A, Q_B \} = ( A Z, J B Z) = (Z , A JB Z) = \frac12 (Z, (AJB - BJA) Z) 
\]
where in the last step we symmetrized the matrix $AJB$.
Thus the multiplication in the induced algebra of symmetric matrices of even dimension is defined by 
\nfootnote{did we want to loose the facto 1/2 here?}
\begin{equation} \label{eqn:Algebra}
    A * B =   AJB - BJA  =  2 [ AJB ]_{sym}\,.
\end{equation}
where $[U]_{sym} = \frac12 ( U + U^t)$ gives the symmetric part of a matrix.
Obviously we have $A * B = -B * A$, so it is in fact a Lie algebra.
The mapping $A = \tilde A J^{-1}$ (or $A = J^{-1} \tilde A$) turns the algebra of symmetric matrices with multiplication $*$ 
into the algebra of Hamiltonian matrices of the form $\tilde A = JA$ (with $A$ symmetric) and algebra multiplication 
given by the standard commutator, $\widetilde{A*B} = J(A*B) =  JA JB - JB JA =  [ JA, JB]  = [ \tilde A, \tilde B]$.
Thus the algebra of symmetric $m\times m$ matrices with multiplication $*$  is isomorphic to the 
symplectic algebra $\sp(2m,J)$. Hence the set of quadratic forms with the standard Poisson bracket 
is a Lie-Poisson algebra of $\sp(2m,J)$. We include the sympectic matrix $J$ in the notation to emphasise
that this can be done for arbitrary symplectic structure.
\end{proof}

We will show that the set of quadratic invariants of the $G(d)$ symmetry of the $n$-body problem is a sub-algebra of this algebra.

Denote the phase space variables in  $\R^{2nd}$ by $Z =  (\q_1, \dots, \q_n, \v_1, \dots \v_n)$ 
(where $\q_i$ represents the $d$ components of the vector $\q_i$, similarly for $\v_i$)
with the modified Poisson structure $\{ (\q_i)_k, (\v_j)_l \} = \delta_{ij} \delta_{kl} / m_j$.
The corresponding symplectic matrix is denoted by 
\begin{equation} \label{eqn:JM}
   J_{nd} =  \begin{pmatrix} 0  & M_{nd}  \\ -M_{nd} & 0 \end{pmatrix}
\end{equation}
with diagonal matrix $M_{nd}$ containing the inverse masses.
Now write two symmetric matrices $A$ and $B$ in block form as
\[
   A =  \begin{pmatrix} R_a & W_a \\ W^t_a & P_a \end{pmatrix}, \quad
   B =  \begin{pmatrix} R_b & W_b \\ W^t_b & P_b \end{pmatrix}, \quad 
\]
with symmetric blocks  $R$ and $P$, and off-diagonal block $W$ each of size $nd \times nd$.
%
%
Then the algebra multiplication of $A$ and $B$ induced by the Poisson bracket 
as given in \eqref{eqn:Algebra} with symplectic matrix \eqref{eqn:JM}
can be explicitly written as
\begin{equation} \label{eqn:AstarBblocks}
   A *  B = \begin{pmatrix} 2  [ W_a M_{nd} R_b - W_b M_{nd} R_a ]_{sym} & [ W_a,  W_b]_M + R_b M_{nd} P_a - R_a M_{nd} P_b \\
          [ W_a, W_b]_M^t + P_a M_{nd} R_b - P_b M_{nd} R_a  & 2[ P_a M_{nd} W_b - P_b M_{nd} W_a]_{sym}
   \end{pmatrix} 
\end{equation}
where $[U,V]_M = UM_{nd}V - VM_{nd}U$ is a ``twisted'' commutator of matrices.

The invariant quadratic forms 
have special structure for two reasons. 
First, there is some redundancy because all the components
of the vectors $\q_i$ and $\v_i$ are treated in the same way. 
Second, there is special structure because the quadratic forms are invariant
under translations and boosts.

Before we describe this structure we remark that it is 
possible to work with momenta instead of with velocities.
By the non-symplectic scaling $\v_i \to \v_i  m_i$ the non-standard symplectic structure $J_{nd}$ 
can be transformed into the standard symplectic matrix $J = T^{-1}  J_{nd} T^{-t}$  where $T = \diag( 1, M_{nd})$. 
The quadratic form $Q_A$ is transformed into $Q_B$ with $B = T^t A T$.
These transformed quadratic forms form an algebra derived from the standard Poisson structure
with standard symplectic matrix $J$. 
However, now the quadratic forms depend on the masses while the new $J$ is independent of the masses.
We prefer to work with velocities since the mass dependence of $J_{nd}$ is simpler than the mass dependence of the transformed 
quadratic form.


Because the components of $\q_i$ and $\v_j$ are all treated in the same way in the invariant quadratic forms the
dimension $d$ does not play a role, 
which is the statement of the next Lemma.
\begin{lemma}
The sub-algebra of invariant quadratic forms is independent of the spatial dimension $d$.
\end{lemma}
\begin{proof}
The ordering of variables used is such that 
the components of vectors are consecutive entries in $Z$.
Since all our invariant quadratic forms come from forming scalar products of differences 
of vectors all the $d$ components of vectors are treated in the same way. 
Thus the matrix $A$ of an invariant quadratic form $Q_A$ has the form 
$A = \hat A \otimes \id_d$ where $\id_d$ is the $d$-dimensional identity matrix
and $\otimes$ denotes the Kronecker product.
\nfootnote{maybe earlier mention that $T^*\R^{nd} = \R^d \otimes \R^{2n} $}
Note that (as a result of the chosen ordering of variables) 
also $J_{nd}$ can be written using the Kronecker product, 
namely $J_{nd} = \hat J \otimes \id_d$. 
Because of the general identity
\[
    (\hat A \otimes \id) (\hat B \otimes \id ) = ( \hat A \hat B \otimes \id)  
\]
the dimension $d$ drops out:
\[
    (\hat A \otimes \id) *  (\hat B \otimes \id ) =2 [  \hat A \hat J \hat B \otimes \id ]_{sym}
    = 2 [  \hat A \hat J \hat B ]_{sym}  \otimes \id \,.
\]
Thus we can define an induced algebra 
\begin{equation} \label{eqn:AstarB}
   \hat A * \hat B = ( \hat A\hat J\hat B - \hat B\hat J\hat A), \quad
     \hat J =  \begin{pmatrix} 0  & \hat M   \\ - \hat M & 0 \end{pmatrix}, \quad
     \hat M = {\rm diag}\left( \frac{1}{m_1}, \dots, \frac{1}{m_n}\right) 
\end{equation}
as before, except that now the quadratic forms have $2n\times 2n$ matrices instead of  $2nd \times 2nd$ matrices.
\end{proof}
It should be noted that the dimension reduced $2n \times 2n$ matrices have the same 
composition law in block form as stated in \eqref{eqn:AstarBblocks}, except that $M_{nd}$ is replaced by $\hat M$.


The more interesting structure of the sub-algebra comes from the fact
the invariant quadratic forms  
are invariant under the symmetry group operations of translations and  boosts.
%
%
There are three basic types of invariants: $\q_{ij} \cdot \q_{kl}$,  $\v_{ij} \cdot \v_{kl}$, and  $\q_{ij} \cdot \v_{kl}$.
The corresponding matrix blocks are $\hat R$, $\hat P$, and $\hat W$, respectively.
The symmetric matrix $\hat R$ corresponding to  $\q_{ij} \cdot \q_{kl}$ has non-zero entries $+1$ at $ik$,  $jl$, $ki$,  $lj$
and $-1$ at $il$, $jk$, $li$, $kj$, so that $\hat R$ is a symmetric Laplacian matrix (row sums are zero). The same set of entries is found 
in $\hat P$ for $\v_{ij} \cdot \v_{kl}$. For the mixed invariant the situation is slightly different. The non-zero entries in $\hat W$ 
corresponding to $\q_{ij} \cdot \v_{kl}$ are $+1$ at $ik$, $jl$ and $-1$ at $il$, $jk$. 
Unless either $i=k$ and $j=l$ or $i=l$ and $j=k$ the matrix $\hat W$ is not symmetric. 
No choice of indices in  $\q_{ij} \cdot \v_{kl}$ gives an antisymmetric $\hat W$.

A quadratic form $Q_A(Z)$ is called {\em shift invariant} if it is invariant under translations and boosts. 
After passing from $nd$ to $n$ dimensions this simply means that adding the same multiple of $(1,0)$ 
to all pairs $(\q_i, \v_i)$ (translations), and adding the same multiple of $(t,1)$ to all pairs $(\q_i,\v_i)$  (boosts) does 
not change the value of the invariant form. 
Define the vector ${\bf t}_1$ with the first $n$ components equal to 1, and the remaining $n$ components equal to 0,
and the vector ${\bf t}_2$ with the first $n$ components equal to 0, and the remaining $n$ components equal to 1.
Shift invariant quadratic forms satisfy $Q_A(Z + {\bf t}_1) = Q_A(Z)$ 
and $Q_A(Z + {\bf t}_2) = Q_A(Z)$.
Accordingly the matrix $A$ of a shift invariant quadratic form has the vectors ${\bf t}_1$ and ${\bf t}_2$ in its kernel. 
The specific form of the two vectors now implies that $A$ is $2\times 2$ {\em block Laplacian}.
Here Laplacian denotes a matrix that has a vector with all $1$'s in its kernel, 
or, equivalently, a matrix that has all row sums equal to zero.

The main observation is that the symmetric block Laplacian matrices 
form a sub-algebra of all symmetric matrices under compositions $*$.
Note that the diagonal blocks $\hat R$ and $\hat P$ are symmetric, while the off diagonal 
block in general is not. It can be decomposed into symmetric Laplacian part $\hat S$ 
and antisymmetric Laplacian part $\hat D$.
We could treat the off-diagonal block $\hat W = \hat S + \hat D$ without splitting it into 
symmetric and antisymmetric part. 



\begin{lemma} \label{lem:blockLap}
The symmetric $2n \times 2n$ matrices of the form
\[
\begin{pmatrix}
\hat R & \hat S + \hat D \\ \hat S - \hat D & \hat P
\end{pmatrix}
\]
where $\hat R$, $\hat P$, and $\hat S$ are $n\times n$ symmetric Laplacian matrices, 
and $\hat D$ is antisymmetric Laplacian
form a sub-algebra of the symmetric $2n \times 2n$ matrices under  composition \eqref{eqn:AstarB}.
\end{lemma}
\begin{proof}
By definition, a (symmetric or antisymmetric) matrix is Laplacian if the vector $(1, \dots, 1)^t$ is 
in the kernel of the matrix. 
From this definition it is clear that Laplacian matrices form a subgroup under matrix multiplication.
The algebra multiplication only involves matrix multiplication and addition of  blocks, and 
thus  $\hat A*\hat B$ is block Laplacian if $\hat A$ and $\hat B$ are block Laplacian. 
In particular the symmetric and antisymmetric part of the off-diagonal 
block can again be written explicitly in terms of matrix multiplications and addition of blocks.
Therefore the symmetric and antisymmetric part of the off-diagonal block are both again
Laplacian.
\end{proof}

We already recalled in Lemma~\ref{lem:SP} that the algebra of quadratic $2m\times 2m$ forms is isomorphic to $\sp(2m)$.
We just showed in Lemma~\ref{lem:blockLap} that the block-Laplacian matrices form a sub-algebra.
This bring us to the following theorem:

\begin{theorem} \label{thm:spfin}
The Algebra of $G(d)$ invariant quadratic forms of the  $n$-body problem in $\R^d$ is 
isomorphic to $\sp(2n-2)$ and the invariant quadratic forms are given by $2\times 2$ 
block-Laplacian matrices.
\end{theorem}
\begin{proof}
In order to show that the sub-algebra is isomorphic to $\sp(2n-2, \hat J)$ we show 
that this $\sp(2n-2)$ consists of the the symplectic mappings of a certain $2n-2$ dimensional 
symplectic subspace of $\R^{2n}$.
\nfootnote{maybe it is strange to keep the $J$ in the notation of $\sp$, since all such spaces are symplectomorphic?}
The subspace $U = \mathrm{span}\{ {\bf t}_1, {\bf t}_2 \}$ is a 2-dimensional symplectic subspace of $\R^{2n}$,
since the symplectic form restricted to $U$ is non-degenerate, namely 
$( {\bf t}_1, \hat J{\bf t}_2) = \sum 1/m_i \not = 0$.
Consequently the symplectic orthogonal complement 
$U^\omega$ for the symplectic form $\omega = (\cdot, \hat J \cdot)$
is a symplectic subspace of $\R^{2n}$ of dimension $2n-2$, 
see, e.g.~\cite{daSilva01}.
\nfootnote{use $\omega$ to talk about quadratic form? See Tabachnikov.}
We already recalled in Lemma~\ref{lem:SP} that $\sp(2m)$ is isomorphic to the the algebra of quadratic forms of $\R^{2m}$
with composition~$*$ induced by the Poisson bracket. According to Lemma~\ref{lem:blockLap} block-Laplacian matrices 
whose kernel is the subspace $U$ form a sub-algebra of this algebra. Hence by restriction to the 
symplectic subspace $U^\omega$ this defines an algebra of quadratic forms isomorphic to $\sp(2n-2)$.
\end{proof}

\noindent
As mentioned earlier the space of invariant quadratic forms has dimension $(2n-1)(n-1)$, which is exactly the dimension of $\sp(2n-2)$.

\subsection*{Dual Pair}

As discussed in \cite{LMS93} the momentum map of the $O(d)$ action on $\R^{2nd}$ 
and the $Sp(2n)$ action with Lie algebra $\sp(2n)$ identified with invariant quadratic forms
are a dual pair. If we replace $\R^{2nd}$ by the the vector space of difference vectors
of dimension $2(n-1)d$ the same construction gives a dual pair for the action of $O(d)$ 
and $Sp(2n-2)$ on difference vectors, or, similarly for the action of $G(d)$ and $Sp(2n-2)$
on the original space $\R^{2nd}$. 
However, explicitly constructing the action of the symplectic group is not so simple in the present case.

\section{Structure of the reduced Poisson bracket}
\label{sec:BasisPoi}

We already mentioned a possible basis for the vector space of invariant quadratic forms  which consists of the entries of the 
Gram matrix of the $2n-2$ vectors $\q_{ij}$ and $\v_{ij}$ for fixed $j$, $i \not = j$. 
Such a basis may be useful if mass $j$ is much bigger than all others.
We are more interested in the general case and are now going to
describe a ``nice'' basis which makes the equations of motion simple
and symmetric. 

As part of the basis for the quadratic invariants we choose the mutual distances squared $\rho_{ij} = ||\q_{ij}||^2$
and the mutually relative speeds squared $\nu_{ij} = ||\v_{ij}||^2$, since these are needed in order to write 
the Hamiltonian $H_c$ in a simple way.
As additional invariants we choose the $n(n-1)/2$ scalar products $\sigma_{ij} = \q_{ij} \cdot \v_{ij}$, $i < j \le n$.
We write $\rho$, $\nu$, $\sigma$ for the column vectors with entries $\rho_{ij}$, $\nu_{ij}$, $\sigma_{ij}$, $i < j \le n$,
respectively.
Now we are still missing $(n-1)(n-2)/2$ invariants to complete the basis of dimension $(2n-1)(n-1)$.
A possible choice is the upper right triangle in the Gram matrix with entries
$ \q_{1j} \cdot \v_{1k}$, $j > k$, depending on the choice of basis of difference vectors, 
but this does not give a nice symmetry in the resulting Poisson bracket.
For now denote any choice that completes the basis of quadratic invariants by $\delta_{ij}$, $i <  j < n$.
For $n = 3$ there are altogether $10 = 3+3+3+1$ independent invariants.
For $n=3$ we will see that $\q_{23} \cdot \v_{13} - \q_{13} \cdot \v_{23}$ is a good choice for $\delta$.
For $n=2$ there is no $\delta$, while for $n=4$ there are three $\delta$s.
%
%

There is a natural block-structure in the  Poisson structure matrix $B$ that arrises by pairing variables 
of the 4 ``types'' $(\rho, \nu, \sigma, \delta)$. The structure of these blocks can be computed 
from the algebra operation in block form \eqref{eqn:AstarBblocks}. 
Symbolically we will denote these blocks by $\{ \rho, \rho \} $, $\{ \rho, \nu \} $ etc.
It is easy to see that $\{ \rho, \rho \}  = 0$ since $ P_a =  P_b =  W_a =  W_b = 0$ implies $ A* B  = 0$, 
similarly $\{ \nu, \nu \}  = 0 $.
Moreover, all blocks can be written in terms of certain fundamental linear functions, 
e.g. $\{ \rho, \sigma \}$ is a linear function of $\rho$ 
while $\{ \sigma, \nu \}$ is that {\em same} linear function evaluated on $\nu$. 
This arrises because for $\rho$, $\nu$ and $\sigma$ we have chosen the same
basis, more precisely the matrix $R$ that corresponds to,  say, $\rho_{ij}$ is the 
same as the matrix $P$ that corresponds to $\nu_{ij}$ and it is also the same 
as the matrix $S$ that corresponds to $\sigma_{ij}$.
This gives another justification for our particular choice of basis of 
quadratic invariant functions.
In this way the whole structure is built from only four linear functions,
only three of which appear in the Hamiltonian vectorfield:
Since the Hamiltonian is independent of $\delta$ the block $\{\delta, \delta\}$ does not appear.
The complete structure is described by the following theorem:

\begin{theorem} \label{thm:PBblocks}
The Poisson structure matrix $B$ for the invariant variables 
$( \rho, \nu, \sigma, \delta)$ has the block form
\[
B = \begin{pmatrix}
0 & 2(L(\sigma) - \Delta) & L(\rho) &  v(\rho) \\
. & 0 & - L(\nu) &  v(\nu) \\
. & . & \Delta &  v(\sigma)\\
. & . & . & \Sigma
\end{pmatrix}
\]
where $\Delta = \Delta( \delta)$ and $\Sigma = \Sigma(\sigma)$ and all four matrix-valued functions 
$L, v, \Delta, \Sigma$
are linear in their arguments
and have coefficients that are of degree $-1$ in the masses $m_i$. In addition $L$ is symmetric,
while $\Delta$ and $\Sigma$ are anti-symmetric.
\end{theorem}

Note that for $n=2$ the blocks $v$ and $\Sigma$ are absent and the block $\Delta = 0$.
For $n=3$ all blocks are present but still $\Sigma = 0$ since it is a $1\times 1$ block and sits in the diagonal.
\nfootnote{can the structure of $B$ be related to the semi-direct product structure of $G(d)$? Or rather $\sp$?}

\begin{proof}
The following argument can be done with the original blocks $R, P, S, D$ or the reduced blocks 
$\hat R, \hat P, \hat S, \hat D$.
Algebra entries of the type $\{ \rho, \sigma \}$ are found from the corresponding quadratic forms using
\[
\begin{pmatrix}
R_a & 0 \\ 0 & 0 
\end{pmatrix}
*
\begin{pmatrix}
0 & S_b \\ S_b & 0  
\end{pmatrix}
=
\begin{pmatrix}
-2[S_b M R_a]_{sym}  & 0 \\  0 & 0
\end{pmatrix} \,.
\]
The entries are located in the upper left block only, so that $\{ \rho, \sigma \}$ is a linear function 
of $\rho$. We define this matrix valued function to be $L(\rho)$.

Similarly  entries of the type $\{ \nu, \sigma \}$ are found from
\[
\begin{pmatrix}
0 & 0 \\ 0 & P_a 
\end{pmatrix}
*
\begin{pmatrix}
0 & S_b \\ S_b & 0  
\end{pmatrix}
=
\begin{pmatrix}
0 & 0 \\  0 & 2[P_a M S_b]_{sym} 
\end{pmatrix} \,.
\]
Now $[P_a M S_b]_{sym} = [S_b M P_a]_{sym}$ since all three matrices in the product are symmetric. 
Thus for $\{\nu, \sigma\}$ up to a minus sign we find the same linear function $L$ as in $\{\rho,\sigma\}$, 
but since the block is located in the lower right this is $-L(\nu)$.

Algebra entries of the type $\{ \rho, \nu \}$ are found from the corresponding quadratic forms by
\[
\begin{pmatrix}
R_a & 0 \\ 0 & 0 
\end{pmatrix}
*
\begin{pmatrix}
0 & 0 \\ 0 & P_b 
\end{pmatrix}
=
\begin{pmatrix}
0 & -  R_a M P_b \\ . & 0
\end{pmatrix} \,.
\]
The off-diagonal block is then decomposed into symmetric and anti-symmetric part and thus 
gives a linear function of $\sigma$ (the symmetric part off-diagonal block) 
and of $\delta$ (the anti-symmetric part of the off-diagonal block).
Thus we see that up to a factor of two \nfootnote{factor 1/2 or 2?}
again we find the linear function $L$ now evaluated at $\sigma$.
The antisymmetric part is linear in $\delta$ and called $\Delta$.
It is given by $-R_a M P_b + P_b M R_a$.

Up to a factor the same function $\Delta$ appears in $\{ \sigma, \sigma \}$ which is computed from 
\[
\begin{pmatrix}
0 & W_a \\ . & 0 
\end{pmatrix}
*
\begin{pmatrix}
0 & W_b \\ . & 0 
\end{pmatrix}
=
\begin{pmatrix}
0 &  [W_a, W_b]_M \\ . & 0
\end{pmatrix}
\]
where $W_a = S_a$ and $W_b = S_b$, hence $S_a M S_b - S_b M S_a$ which is anti-symmetric, which again gives
$\Delta(\delta)$. \nfootnote{up to a minus sign...}

The same matrix operation is used to compute $\{ \delta, \delta \}$ except that now anti-symmetric Laplacian matrices 
$D_a, D_b$ are used so that $\Sigma(\delta)$ is obtained from $D_a M D_b - D_b M D_a$.

The remaining blocks are $\{ \rho, \delta\}$,  $\{ \nu, \delta\}$ and  $\{ \sigma, \delta\}$, which 
all lead to the same linear function $v$ evaluated at $\rho$, $\nu$, $\sigma$, respectively.
The corresponding matrix products are 
$-2[D_b M R_a]_{sym}$, $2[P_a M D_b]_{sym}$, and $[S_a, D_b]_M$ which all define the same function $v$.
\nfootnote{up to signs and factors...}
\end{proof}

We now give the explicit form of the block $L(\tau)$ where $\tau = \rho$, $\nu$, or $\sigma$.
It is convenient to keep using two indices $ij$ where $i< j\le n$ for the components of the vector $\tau$.
Changing the ordering of the components of $\tau$ gives a permutation of $L$.
We denote the entries of $L$ by $L_{ij, kl}$.

\begin{lemma}\label{lem:Lcomp}
The entries in the matrix valued function $L(\tau)$ are given by 
$L_{ij, ij} = 2 \tau_{ij}/\mu_{ij}$ in the diagonal where $1/\mu_{ij} = 1/m_i + 1/m_j$,
$L_{ij, kl} = 0$ if no two indices in $ij$ and $kl$ coincide, and 
the remaining non-zero entries are
$L_{ij,jl} = (\tau_{ij} + \tau_{jl} - \tau_{il})/m_j$.
\end{lemma}
\begin{proof}
A basis for symmetric Laplacian matrices is given by $E_{ij}$ for $i< j< n$ where the 
$ii$ and the $jj$ entry are $+1$, the $ij$ and $ji$ entry are $-1$ and all other entries are zero.

We already showed that $L$ is found three times in $B$, we choose to compute its components 
from the $\{\rho, \sigma\}$ block.
To compute $\{\rho_{ij}, \sigma_{kl} \}$ find the symmetric Laplacian matrices 
$\hat R_{ij}$ and $\hat S_{kl}$ corresponding to $\rho_{ij}$ and $\sigma_{kl}$, respectively. 
Then compute $ 2 [ \hat R_{ij} \hat M \hat S_{kl} ]_{sym} $ and express this as a linear 
combination of symmetric Laplacian matrices $\hat S_{uv}$ with coefficients $c_{uv}$. Then we have found that 
$\{ \rho_{ij}, \sigma_{kl} \} = \sum_{uv} \sigma_{uv} c_{uv}$. 
Of course the coefficients $c_{uv}$ depend on ${ij, kl}$ but this is suppressed in the notation. 
Each entry of the block $\{ \rho, \sigma \} $ is computed in this way, and all together
this defines $L(\sigma)$.

First consider the diagonal of $L$ with $ij = kl$, so that $\hat R_{ij} = \hat S_{kl} = E_{ij}$. Hence the diagonal entries of $L$ are of the form 
$2 [ E_{ij} \hat M E_{ij}]_{sym}  = 2 E_{ij} \hat M E_{ij} = 2 \mu_{ij}^{-1} E_{ij}$, so that in this case the single 
non-zero $c_{uv}$ is $c_{ij} = 2 / \mu_{ij}$, so that the diagonal entries of $L(\sigma)$ are $\sigma_{ij}/\mu_{ij}$.
Next consider the case that the pairs of indices $ij$ and $kl$ have no index in common. 
Then it is easy to see that $E_{ij} M E_{kl} = 0$, so the corresponding entry of $L$ vanishes.
Finally consider the case in which there is exactly one index in common between $ij$ and $kl$,
say $j=k$. By symmetry $E_{ij} = E_{ji}$ the other cases with an index in common can be reduced to this case.
Then $2[ E_{ij} M E_{jl}]_{sym} = (E_{ij} + E_{jl} - E_{il})/m_j$. This gives the off-diagonal 
non-zero entries of $L$.

\end{proof}

Note that the $L(\tau)$ block is independent of the choice of basis for $\delta$. 
The explicit form of the further blocks $\Delta$, $v$ and $\Sigma$ depends on the choice of basis for $\delta$
and we give explicit formulae for $n=3$ and $n=4$ in the following subsections.

Once a basis for the antisymmetric Laplacian matrices is chosen the computation of the 
components of $\Delta, v, \Sigma$ proceeds in a way similar to Lemma~\ref{lem:Lcomp}.
The building blocks for such a basis are quadratic forms 
$C_{ij, kl} = \q_{ij} \cdot \v_{kl} - \v_{ij} \cdot \q_{kl} $  with $k=j$ with corresponding 
antisymmetric Laplacian matrix $\hat D = E_{ij} E_{jl} - E_{jl} E_{ij}$.
However, except for $n=3$ there are too many such matrices to form a basis.
For example the entry $\Delta_{ij,kl}$ is given by  the antisymmetric part of the product 
$E_{ij} \hat M E_{kl}$. As before the results is zero if no indices are in common. 
In addition from anti-symmetry the result is also zero if both double-indices are the same.
Hence the only non-zero entries are $E_{ij} \hat M E_{jl} - E_{jl} \hat M E_{ij}$
which now need to be expressed in terms of the basis of antisymmetric Laplacian matrices.

\subsection*{Three bodies}

For $n=3$ the 10 invariants are 
$\rho = ( ||\q_{23}||^2, ||\q_{13}||^2, ||\q_{12}||^2)^t$, 
$\nu =( ||\v_{23}||^2, ||\v_{13}||^2, ||\v_{12}||^2)^t$, 
$\sigma =( \q_{23} \cdot \v_{23} ,  \q_{13} \cdot \v_{13}, \q_{12} \cdot \v_{12})^t$, 
and the single entry $\delta =  (\q_{23} \cdot \v_{31} -  \v_{23} \cdot  \q_{31} )$.
The corresponding antisymmetric Laplacian matrix $\hat D$ of the quadratic 
form $\delta$ is the unique (up to scaling) antisymmetric Laplacian matrix 
in dimension three.
The Poisson structure matrix is of the general form described in Theorem~\ref{thm:PBblocks}
where $\Sigma = 0$.
%
The building blocks of the Poisson structure matrix are
\[
   \Delta(\delta) =  \delta \begin{pmatrix} 0 & 1/m_3 &  -1/m_2 \\ . & 0 & 1/m_1 \\  . &  . & 0  \end{pmatrix}
\]
where $\Delta^t = -\Delta$ and
\[
     L( \tau_{23},\tau_{13},\tau_{12} ) = \begin{pmatrix} 
     2 \tau_{23}/\mu_{23} & (\tau_{23} + \tau_{13} - \tau_{12})/m_3 & (\phantom{-}\tau_{23} - \tau_{13} + \tau_{12})/m_2 
     \\ . & 2 \tau_{13}/\mu_{13} & (-\tau_{23} + \tau_{13} + \tau_{12})/m_1 \\ 
   . & . & 2 \tau_{12} / \mu_{12}  \end{pmatrix} \,,
\]
where $1/\mu_{ij} = 1/m_i + 1/m_j$ and $L^t = L$. The column vector $v$ is
\[
   v( \tau_{23}, \tau_{13}, \tau_{12} ) = \begin{pmatrix}
 {{(\phantom{-}\tau_{23}+\tau_{13}-\tau_{12})}/{m_2}-(\phantom{-}\tau_{23}-\tau_{13}+\tau_{12})}/{m_3} \\
 {{(-\tau_{23}+\tau_{13}+\tau_{12})}/{m_3}-(\phantom{-}\tau_{23}+\tau_{13}-\tau_{12})}/{m_1} \\
 {{(\phantom{-}\tau_{23}-\tau_{13}+\tau_{12})}/{m_1}-(-\tau_{23}+\tau_{13}+\tau_{12})}/{m_2}
\end{pmatrix} \,.
\]

The bracket has two Casimirs. One Casimir is the
determinant of the Gram matrix of the vectors 
$(\q_{23}, \q_{13}, \v_{23}, \v_{13})$.
In terms of invariants the symmetric Gram matrix is
\[
  2 G = \begin{pmatrix} 2 \rho_{23} & \rho_{12} - \rho_{13} - \rho_{23} & 
  		2 \sigma_{23} & \delta  + \sigma_{12} - \sigma_{13} - \sigma_{23} \\
  . & 2 \rho_{13} & -\delta +  \sigma_{12} - \sigma_{13} - \sigma_{23} & 2 \sigma_{13} \\
  . & . & 2 \nu_{23} & \nu_{12} - \nu_{13} - \nu_{23} \\ 
  . & . & . & 2 \nu_{13}  \end{pmatrix} \,.
\]
A different choice of basis for the difference vectors, e.g.\ $(\q_{12}, \q_{31}, \v_{12}, \v_{23})$,
gives a different Gram matrix. However, the determinant of this matrix in terms of the invariants is the same.
%
The Gram determinant is homogeneous of degree 4 in the invariants.
The surface $\det G = 0$ is an example of a (linear) determinantal variety \cite{Harris92}.
For $d=3$ we necessarily have $\det G = 0$ since the 4-volume spanned 3-vectors vanishes.
Hence for $d=3$ the condition $\det G = 0$ is a relation between the quadratic invariants.

The other Casimir is the total angular momentum $|| {\bf L}_c||^2$ with respect to the centre of mass
which for $n=3$ can be written in terms of the invariants as
\[
  ||{\bf L}_c||^2 = \sum_{i<j} \frac{m_i^2 m_j^2}{\mtot^2} ( \rho_{ij} \nu_{ij} - \sigma_{ij}^2)
  + \frac{ m_1 m_2 m_3}{2 \mtot} \left( \delta_{12}^2+ \sum_{i<j} \frac{m_k}{\mtot} (( \rho_s -2 \rho_{ij}) (\nu_s - 2\nu_{ij}) - (\sigma_s - 2\sigma_{ij})^2) \right)
\]
where $\rho_s = \sum_{i<j} \rho_{ij}$, 
$\nu_s = \sum_{i<j} \nu_{ij}$,
$\sigma_s = \sum_{i<j} \sigma_{ij}$.
This Casimir is homogeneous quadratic in the invariants and hence defines a quadric, 
which is non-singular  whenever the value of $||{\bf L}_c||^2$ is positive.

\subsection*{Four bodies} 

For $n=4$ there are 21 invariants. 
The invariants are ordered as $(\tau_{12}, \tau_{13}, \tau_{14}, \tau_{23}, \tau_{24}, \tau_{34})$ for 
$\tau = \rho, \nu, \sigma$. 
We choose $\delta = ( C_{12,43}, C_{23, 41}, C_{24, 31})$ where $C_{ij, kl} = 
\q_{ij} \cdot \v_{kl} - \v_{ij} \cdot \q_{kl} $ as a basis for the antisymmetric Laplacian invariant quadratic forms.
The $6\times 6$-matrix valued function $L(\tau_{ij})$ is given by 
\[
\left(
\begin{array}{cccccc}
 \frac{2 \tau _{1,2}}{\mu _{1,2}} & \frac{\tau _{1,2}+\tau _{1,3}-\tau _{2,3}}{m_1} & \frac{\tau _{1,2}+\tau _{1,4}-\tau
   _{2,4}}{m_1} & \frac{\tau _{1,2}-\tau _{1,3}+\tau _{2,3}}{m_2} & \frac{\tau _{1,2}-\tau _{1,4}+\tau _{2,4}}{m_2} & 0 \\
 . 
 & \frac{2 \tau _{1,3}}{\mu _{1,3}} & \frac{\tau _{1,3}+\tau _{1,4}-\tau
   _{3,4}}{m_1} & \frac{-\tau _{1,2}+\tau _{1,3}+\tau _{2,3}}{m_3} & 0 & \frac{\tau _{1,3}-\tau _{1,4}+\tau _{3,4}}{m_3}
   \\
. & . & \frac{2 \tau_{1,4}}{\mu _{1,4}} & 0 & \frac{-\tau _{1,2}+\tau _{1,4}+\tau _{2,4}}{m_4} & \frac{-\tau _{1,3}+\tau _{1,4}+\tau
   _{3,4}}{m_4} \\
. & . & 0 & \frac{2 \tau_{2,3}}{\mu _{2,3}} & \frac{\tau _{2,3}+\tau _{2,4}-\tau _{3,4}}{m_2} & \frac{\tau _{2,3}-\tau _{2,4}+\tau _{3,4}}{m_3}
   \\
. & 0 &
 . & .  & \frac{2 \tau _{2,4}}{\mu _{2,4}} & \frac{-\tau _{2,3}+\tau _{2,4}+\tau_{3,4}}{m_4} \\
 0 & 
. & . & . & .  & \frac{2 \tau _{3,4}}{\mu_{3,4}}
\end{array}
\right) \,.
\]
The $6\times 3$-matrix valued function $v(\tau_{ij})$ is given by
\[
 \left(
\begin{array}{ccc}
 -\frac{\tau _{1,3}-\tau _{1,4}-\tau _{2,3}+\tau _{2,4}}{\mu _{1,2}} & \frac{\tau _{1,2}-\tau _{1,3}+\tau
   _{2,3}}{m_1}+\frac{-\tau _{1,2}-\tau _{1,4}+\tau _{2,4}}{m_2} & \frac{-\tau _{1,2}-\tau _{1,3}+\tau
   _{2,3}}{m_2}+\frac{\tau _{1,2}-\tau _{1,4}+\tau _{2,4}}{m_1} \\
 \frac{\tau _{1,2}+\tau _{1,3}-\tau _{2,3}}{m_3}+\frac{-\tau _{1,3}+\tau _{1,4}-\tau _{3,4}}{m_1} & \frac{\tau _{1,2}-\tau
   _{1,3}-\tau _{2,3}}{m_1}+\frac{\tau _{1,3}+\tau _{1,4}-\tau _{3,4}}{m_3} & \frac{\tau _{1,2}-\tau _{1,4}-\tau
   _{2,3}+\tau _{3,4}}{\mu _{1,3}} \\
 \frac{-\tau _{1,2}-\tau _{1,4}+\tau _{2,4}}{m_4}+\frac{-\tau _{1,3}+\tau _{1,4}+\tau _{3,4}}{m_1} & \frac{\tau
   _{1,2}-\tau _{1,3}-\tau _{2,4}+\tau _{3,4}}{\mu _{1,4}} & \frac{\tau _{1,2}-\tau _{1,4}-\tau _{2,4}}{m_1}+\frac{\tau
   _{1,3}+\tau _{1,4}-\tau _{3,4}}{m_4} \\
 \frac{-\tau _{1,2}+\tau _{1,3}-\tau _{2,3}}{m_3}+\frac{\tau _{2,3}-\tau _{2,4}+\tau _{3,4}}{m_2} & -\frac{\tau
   _{1,2}-\tau _{1,3}-\tau _{2,4}+\tau _{3,4}}{\mu _{2,3}} & \frac{-\tau _{1,2}+\tau _{1,3}+\tau _{2,3}}{m_2}+\frac{-\tau
   _{2,3}-\tau _{2,4}+\tau _{3,4}}{m_3} \\
 \frac{\tau _{1,2}-\tau _{1,4}+\tau _{2,4}}{m_4}+\frac{\tau _{2,3}-\tau _{2,4}-\tau _{3,4}}{m_2} & \frac{-\tau _{1,2}+\tau
   _{1,4}+\tau _{2,4}}{m_2}+\frac{-\tau _{2,3}-\tau _{2,4}+\tau _{3,4}}{m_4} & -\frac{\tau _{1,2}-\tau _{1,4}-\tau
   _{2,3}+\tau _{3,4}}{\mu _{2,4}} \\
 \frac{\tau _{1,3}-\tau _{1,4}-\tau _{2,3}+\tau _{2,4}}{\mu _{3,4}} & \frac{\tau _{1,3}-\tau _{1,4}-\tau
   _{3,4}}{m_3}+\frac{\tau _{2,3}-\tau _{2,4}+\tau _{3,4}}{m_4} & \frac{-\tau _{1,3}+\tau _{1,4}-\tau
   _{3,4}}{m_4}+\frac{-\tau _{2,3}+\tau _{2,4}+\tau _{3,4}}{m_3}
\end{array}
\right) \,.
\]
The $6\times 6$-matrix valued function $\Delta( \delta)$ is given by
\[
\left(
\begin{array}{cccccc}
 0 & -\frac{\delta _1+\delta _2+\delta _3}{2 m_1} & -\frac{-\delta _1+\delta _2+\delta _3}{2 m_1} & \frac{\delta _1+\delta
   _2+\delta _3}{2 m_2} & \frac{-\delta _1+\delta _2+\delta _3}{2 m_2} & 0 \\
 . & 0 & \frac{\delta _1+\delta _2-\delta _3}{2 m_1} & -\frac{\delta _1+\delta
   _2+\delta _3}{2 m_3} & 0 & -\frac{\delta _1+\delta _2-\delta _3}{2 m_3} \\
 . & . & 0 & 0 & -\frac{-\delta_1+\delta _2+\delta _3}{2 m_4} & \frac{\delta _1+\delta _2-\delta _3}{2 m_4} \\
. & .  & 0 & 0 & -\frac{\delta_1-\delta _2+\delta _3}{2 m_2} & \frac{\delta _1-\delta _2+\delta _3}{2 m_3} \\
  . & 0 & 
  . & . &  0 & -\frac{\delta _1-\delta _2+\delta _3}{2 m_4} \\
 0 & 
  . & . & . & . & 0
\end{array}
\right) \,.
\]
Finally the $3\times 3$-matrix valued function $\Sigma(\delta)$  is given by
\[
 \left(
\begin{array}{ccc}
 0 & 
   \frac{\delta _1-\delta _2+\delta _3}{2 m_1}
 -\frac{\delta _1+\delta _2-\delta _3}{2 m_2}
 +\frac{-\delta _1+\delta_2+\delta _3}{2 m_3}
 +\frac{\delta _1+\delta _2+\delta _3}{2 m_4} 
 & 
   \frac{\delta _1-\delta_2+\delta _3}{2 m_1} 
 -\frac{\delta _1+\delta _2-\delta _3}{2 m_2}
 -\frac{-\delta _1+\delta _2+\delta _3}{2 m_3}
 -\frac{\delta _1+\delta _2+\delta _3}{2 m_4}
 \\
  .  
 &  0 & 
   \frac{\delta _1-\delta _2+\delta _3}{2 m_1}
 +\frac{\delta _1+\delta _2-\delta _3}{2 m_2}
 -\frac{-\delta _1+\delta _2+\delta _3}{2 m_3}
 +\frac{\delta _1+\delta_2+\delta _3}{2 m_4}
  \\
.  
& .  
   & 0
\end{array}
\right) \,.
\]
The rank of this Poisson structure matrix is 18, so there are 3 Casimirs:
the total angular momentum, the Gram determinant of the $6 \times 6$ Gram matrix 
and a third Casimir which can be obtained from minors of the Gram determinant.
\nfootnote{For $d=3$ the fully reduced system has dimension 14 only, so there must be 4 additional conserved 
quantities, probably related to the minors of $G$, since the $6\times 6$ Gram matrix will have 
a kernel of dimension 3 for spatial 3D motion. Hence all the $5\times 5$ and all $4\times 4$ minors 
will vanish. Also for $d=4$ the $5\times 5$ minors vanish, while for $d=5$ only $\det G = 0$ and only 
for $d \ge 6$ we may have $\det G > 0$.}
\nfootnote{did we compute this Casimir somewhere? Did not Albouy give a formula?}

\section{Poisson Integrator}

A Poisson map $\phi : M \to N$ satisfies
\[
    \{ f \circ \phi, g \circ \phi \}_M = \{ f, g \}_N \circ \phi \,.
\]
Let $B$ be the structure matrix of the Poisson structure, and let $\phi$ be a map
from $M$ to itself, as it arrises when $\phi$ is given by the flow of Poisson differential equations.
Linearising the definition in this case gives
\[
    D\phi(x) B(x) D\phi(x)^t = B(\phi(x)) \,.
\]
An integrator with this property preserves the geometric structure of the flow. 
A symplectic map is a special case for which $B(x) = J^{-1}$ is a constant even dimensional antisymmetric matrix.
 

We treat the case $n=2$ first before treating $n=3,4$.
A splitting integrator uses the kinetic and the potential energy separately as 
generators of flows.
The vectorfield $X_K$ generated by $K_c$ is
\[
X_K = 
   \begin{pmatrix}
         \dot \rho_{12} \\ \dot \nu_{12} \\ \dot \sigma_{12} 
   \end{pmatrix}
=
\frac{2}{\mu} \begin{pmatrix}
      0 & 2 \sigma_{12} & \rho_{12} \\
      -2 \sigma_{12} & 0 & - \nu_{12} \\
      -\rho_{12} & \nu_{12} & 0 
   \end{pmatrix} 
   \begin{pmatrix}
         0 \\ \mu/2 \\ 0 
   \end{pmatrix}
   = 
   \begin{pmatrix}
          2 \sigma_{12} \\ 0 \\ \nu_{12} 
   \end{pmatrix}   \,.
\]
Thus $\nu_{12}$ is constant, hence $\sigma_{12}$ is linear in time, 
and hence $\rho_{12}$ is quadratic in time. The solution is 
\[
   \begin{pmatrix}
         \rho_{12}(t) \\ \nu_{12}(t) \\ \sigma_{12}(t)
   \end{pmatrix}
   =
   \begin{pmatrix}
         \rho_{12}(0) +2 t \sigma_{12}(0) + t^2  \nu_{12}(0)  \\ \nu_{12}(0) \\ \sigma_{12}(0) + t \nu_{12}(0)
   \end{pmatrix}
\]
Similarly the vector field $X_V$ generated by $V$ is given by
\[
X_V =  
   \begin{pmatrix}
         \dot \rho_{12} \\ \dot \nu_{12} \\ \dot \sigma_{12} 
   \end{pmatrix}
=
\frac{2}{\mu} \begin{pmatrix}
      0 & 2 \sigma_{12} & \rho_{12} \\
      -2 \sigma_{12} & 0 & - \nu_{12} \\
      -\rho_{12} & \nu_{12} & 0 
   \end{pmatrix} 
   \begin{pmatrix}
         V'(\rho_{12}) \\ 0 \\ 0 
   \end{pmatrix}
   = 
   \frac{2}{\mu} \begin{pmatrix}
          0 \\ -2 \sigma_{12} V'(\rho_{12})  \\ - \rho_{12} V'(\rho_{12}) 
   \end{pmatrix}   \,.
\]
Thus $\rho_{12}$ is constant, hence $\sigma_{12}$ is linear in time, and hence $\nu_{12}$ is quadratic in time.
The solution is
\[
   \begin{pmatrix}
         \rho_{12}(t) \\ \nu_{12}(t) \\ \sigma_{12}(t)
   \end{pmatrix}
   =
   \begin{pmatrix}
        \rho_{12}(0) \\ \nu_{12}(0) - \frac{4}{\mu} t \sigma_{12}(0)V'(\rho_{12}(0)) + \frac{8}{\mu^2} t^2 \rho_{12} V'(\rho_{12}(0))^2 \\ 
        \sigma_{12}(0) -\frac{2}{\mu} t \rho_{12} V'(\rho_{12})
   \end{pmatrix} \,.
\]

We now treat the case of $n > 2$.
For a splitting integrator it is crucial that the part of the Hamiltonian $K_c(\nu)$ and $V(\rho)$ produce integrable flows. 
Write $Y = (\rho, \nu, \sigma, \delta)$. 
Recall that $K_c(\nu)$ is linear in $\nu$ and independent of the other variables. 
Hence $K'_c(\nu)$ is a constant vector. 
The vector field generated by $K_c(\nu)$ is given by 
\[
  \dot Y = X_K = B \cdot \nabla K_c = \begin{pmatrix} 
     2(L(\sigma) + \Delta) K'_c \\
     0 \\
     L(\nu) K'_c \\
     -2v(\nu)^t K'_c
  \end{pmatrix} \,.
\]
Therefore $\nu$ is constant, and hence the derivative of $\sigma$ and $\delta$ is constant, 
so that they integrate to linear functions of $t$.
The remaining equations for $\rho$ thus have a linear function of $t$ 
on the right hand side, and integrate to quadratic functions of $t$.

Using the particular form of $K_c(\nu)$ brings additional simplification. 
Recall that 
$
   K_c(\nu) = \frac{1}{2\mtot} \sum m_i m_j \nu_{ij} \,.
$
Using the formulas from the previous section for $n=3,4$ the constant vector $K'_c$ satisfies the following identities:
\[
    L(\tau) K'_c = \tau, \quad K'_c v(\tau)  = 0, \quad \tau = \rho, \nu, \sigma, \qquad \Delta K'_c = 0 \,.
\] 
As a result the vector field $X_K$ simply becomes 
\[
  \dot \rho = 2 \sigma, \dot \nu = 0, \dot \sigma = \nu, \dot \delta = 0,
\]
which recovers Newton's first law in our coordinate system.
%
%
The explicit solution  
is the Poisson map
\[
  \phi_K^t ( \rho, \nu, \sigma, \delta) = (\rho + 2 t \sigma + t^2 \nu, \nu, \sigma + t \nu, \delta) \,.
\]
Notice that $\phi_K^t$ is linear in the phase space variables. 
\rem{
Hence the Jacobian matrix of $\phi_K^t$ is independent of the initial conditions and for $n=3$ it is given by 
\[
D\phi^t_K = \frac{ \partial \phi_K^t( \rho, \nu, \sigma, \delta) } {\partial (\rho, \nu, \sigma, \delta) } = 
\left(
\begin{array}{cccccccccc}
 1 & 0 & 0 & t^2 & 0 & 0 & 2 t & 0 & 0 & 0 \\
 0 & 1 & 0 & 0 & t^2 & 0 & 0 & 2 t & 0 & 0 \\
 0 & 0 & 1 & 0 & 0 & t^2 & 0 & 0 & 2 t & 0 \\
 0 & 0 & 0 & 1 & 0 & 0 & 0 & 0 & 0 & 0 \\
 0 & 0 & 0 & 0 & 1 & 0 & 0 & 0 & 0 & 0 \\
 0 & 0 & 0 & 0 & 0 & 1 & 0 & 0 & 0 & 0 \\
 0 & 0 & 0 & t & 0 & 0 & 1 & 0 & 0 & 0 \\
 0 & 0 & 0 & 0 & t & 0 & 0 & 1 & 0 & 0 \\
 0 & 0 & 0 & 0 & 0 & t & 0 & 0 & 1 & 0 \\
 0 & 0 & 0 & 0 & 0 & 0 & 0 & 0 & 0 & 1
\end{array}
\right) \,.
\]
}

Using the formulas from the previous section for $n=3,4$ the vector field of the potential $V(\rho)$  is given by
\[
  \dot Y = X_V = B \cdot \nabla V = \begin{pmatrix} 
     0 \\
     -2(L(\sigma) + \Delta) V' \\
     -L(\rho) V' \\
     -2 v(\rho)^t V'
  \end{pmatrix} \,.
\]
Therefore $\rho$ is constant, and thus the vector $V'(\rho)$ is constant as well.
Hence the derivative of $\sigma$ and $\delta$ is also constant, 
and they integrate to linear functions of time.
The remaining equations for $\nu$ thus have a linear function of $t$
on the right hand side, and integrate to quadratic functions of $t$.
The explicit solution is the Poisson map
\[
  \phi_V^t ( \rho, \nu, \sigma, \delta) = (\rho, \nu -  2 t(L(\sigma) + \Delta(\delta))V'(\rho) - t^2(L(a) + \Delta(b))V'(\rho), \sigma + t a  , \delta + t b) \,.
\]
where $a = - L(\rho) V'(\rho) $ and $b = - v(\rho)^t V'(\rho)$.
Using linearity the solution for $\nu$ can be rewritten as
\[
 \nu(t) = \nu - ( L( 2 t \sigma +  t^2 a) + \Delta( 2 t \delta  + t^2 b)) V'(\rho) \,.
\]
The map $\phi_V^t$ is linear in the initial conditions $\nu, \sigma, \delta$, but non-linear in $\rho$, 
unlike $\phi_K^t$, which is linear in all initial conditions.

Both combinations $\phi_V^t \phi_K^t$ and $\phi_K^t \phi_V^t$ are first order Poisson integrators
for the reduced $n$-body problem. However, neither of them is reversible.
Combining the two first order integrators so that the more expensive step $\phi_V^t$ is only 
used once gives a 2nd order integrator
\[
 \Phi^t = \phi_K^{t/2} \circ \phi_V^t \circ \phi_K^{t/2} \,.
\]
By construction the integrator is Poisson and exactly preserves the two Casimirs.
In addition the integrator $\Phi^t$ is reversible, i.e.\ it satisfies $\Phi^t \circ \Phi^{-t} = id$.
This follows from the fact that each individual map $\phi_K^t$ and $\phi_V^t$ is a flow, 
and hence satisfies the flow property $\phi_V^t \circ \phi_V^s = \phi_V^{s+t}$.
A proof that $\Phi$ is a second order integrator can be found in \cite{HaLuWa02}, 
this follows in general for splitting methods.
Moreover, from the building blocks $\phi_K$ and $\phi_V$ also higher order
integrators can be constructed, see \cite{Yoshida90,HaLuWa02} and the references therein.

A fundamental property of the {\em gravitational} $n$-body problem is preserved by this integrator, which 
is the scaling invariance. The Hamiltonian and equations of motion are unchanged when 
time is scaled by $\tau$ and space is scaled by $\lambda$ such that $\lambda^3 = \tau^2$.
The induced scaling of the invariants is $S^\lambda(\rho,\nu, \sigma, \delta) = ( \lambda^2 \rho, \lambda^{-1} \nu, \sqrt{\lambda}\sigma , \sqrt{\lambda} \delta  )$. Now it is easy to check that for both, $\phi_K$ and $\phi_V$, and hence for $\Phi$ we have
\begin{equation} \label{eqn:scaleh}
     \Phi^{t \lambda^{3/2}} \circ S^\lambda = S^\lambda \circ \Phi^t,
\end{equation}
and the Hamiltonian is scaled by $\lambda^{-1}$ (as is $\nu$) when the stepsize is scaled by $\lambda^{3/2}$.
Other homogeneous potentials have similar scaling laws.
Because of this scaling symmetry an orbit of $\Phi^h$ can be mapped to a scaled orbit
with scaled stepsize.
The analogous property in the flow is that periodic orbits appear in families parametrized by the energy, 
and in the gravitational case these families are obtained from the scaling symmetry.
Since the Hamiltonian is not conserved for the map $\Phi^h$ the natural family parameter in the discrete case 
is the stepsize. 
 
\section{Numerical Example: Figure 8 in 18 steps}

\begin{figure}
\centerline{\includegraphics[width=8cm]{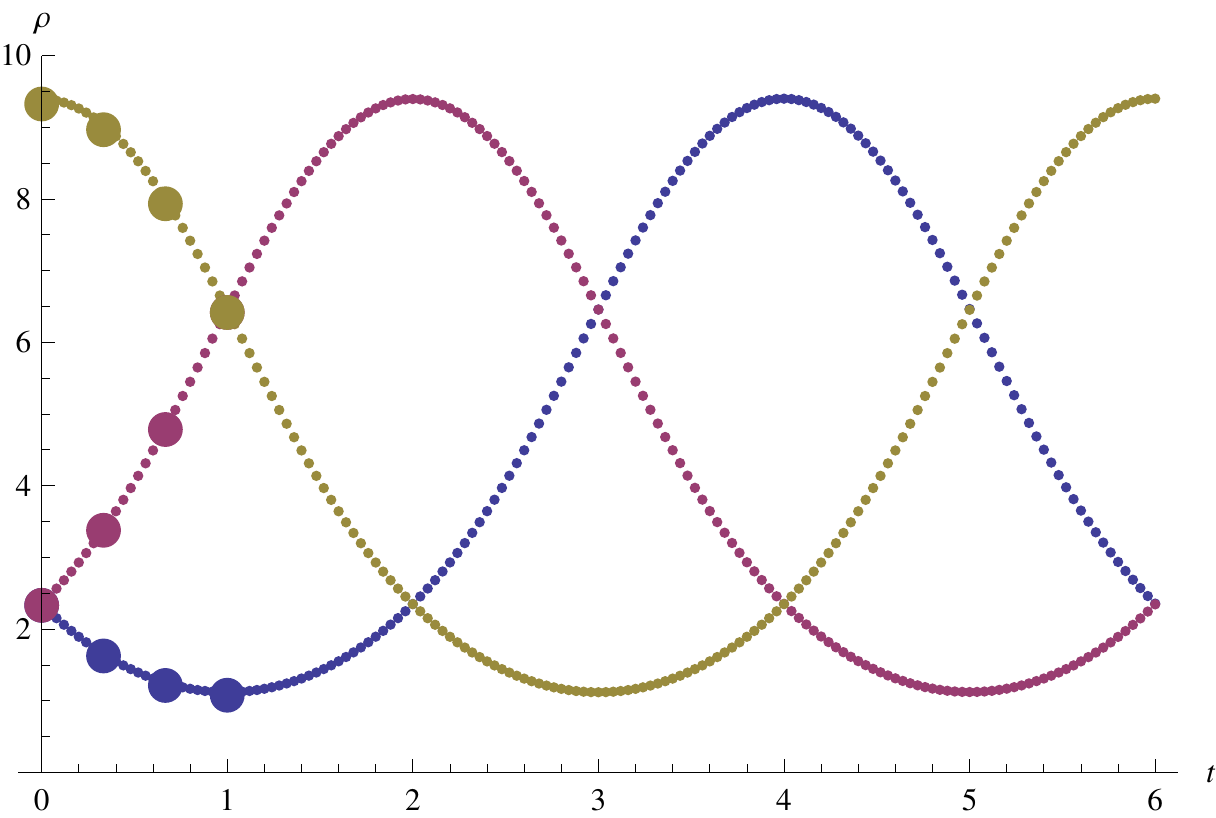}}
\caption{Evolution of $\rho_{12}, \rho_{13}, \rho_{23}$ for the figure 8 discretised with 150 steps of stepsize $h=0.04$ 
(small dots).
Overlaid are the first 3 iterates (big dots) for the figure 8 in 18 steps, starting a the isosceles collinearity and finishing at the non-collinear isosceles configuration. 
This period 18 orbit of $\Phi^h$ with $h=1/3$ 
has the same discrete symmetry and linear stability as the figure 8.
The remaining points of the period 18 orbit can be generated from the discrete symmetry.
Initial conditions for both orbits are chosen so that $T = nh = 6$, which makes their size surprisingly similar.}
\end{figure}

Since a Poisson integrator preserves the Casimirs and the geometric structure of the problem 
exactly, it may be sensible to consider unusually large time steps $h$, and still get qualitatively 
correct results. We use this kind of ultra-discretisation to show that there are periodic orbits
of period 18 of the map $\Phi^h$ that have the same discrete symmetry and the same stability
as the  figure 8 choreography of Chenciner and Montgomery \cite{CM00}.

A numerical integrator is a map with the time step $h$ as a continuous parameter.
Usually $h$ is chosen sufficiently small so that  no essential 
change occurs when $h$ is changed. From the scaling relation \eqref{eqn:scaleh}
we see that in our particular problem orbits of the integrator appear in families.
This reflects the well known scaling symmetry of the $n$-body problem 
with homogeneous potential. Even without such a scaling relation periodic orbits 
of a Hamiltonian flow appear in families locally parametrized by the period or by the value of the Hamiltonian.
Since the integrator in general does not conserve the Hamiltonian the only 
available parameter is the stepsize, and we may expect that locally (i.e.\ ignoring bifurcations) 
periodic orbits appear in families parametrized by the stepsize. 
In the particular case of the gravitational $n$-body problem the integrator
even has a global scaling property given above.

When using a Poincar\'e section to find periodic orbits the energy is fixed and the (continuous time) period 
of the periodic orbit is undetermined. When considering the integrator $\Phi^h$ as a discrete 
dynamical system there is no sense in fixing the energy since it is not conserved.
Instead we look for (discrete) period $n$ orbits for fixed step size $h$ of the integrator 
(and hence fixed period $T=nh$ in the continuum limit). Because of the scaling symmetry 
the step size can be arbitrarily fixed. 

The figure 8 choreography can be discretised with 18 steps only, see Figure~1.
Period 18 amounts to only 3 iterates after factoring out the discrete $D_6/\mathbb{Z}_2$ symmetry (see \cite{CM00}) 
in reduced space.
With such a huge stepsize relative to the period $T=6$ it is not obvious how to identify the discretised figure 8 orbit.
The precise claim is that there exits a period 18 orbit of $\Phi^h$ that has the same discrete symmetries as the Figure 8, 
is  linearly stable, and roughly follows the shape of the figure 8. For comparison a figure 8 discretised with 150 
steps is also shown in Figure~1.
Periodic orbits of $\Phi^h$ with period 6 or 12 with the correct symmetry do exist, but they are not elliptic.
Starting with the collinear configuration, the discrete symmetry forces the initial condition to be of the form 
$\rho_{13}  = \rho_{23}, \nu_{12} = 0, \nu_{13} = \nu_{23}, \sigma_{12} = 0, \sigma_{13} = -\sigma_{23}$, 
and imposing $\det G = ||\mathbf{L}_c||^2 = 0$ in addition gives
$\rho_{12} = 4\rho_{13}$ and $\delta_{12} = 2\sigma_{13}$ so that there are only 3 parameters for orbits with this symmetry. 
Trying to find a symmetric period $6m$ solution requires that the $m$th iterate of this initial condition is 
at the isosceles configuration of the form $\rho_{12} = \rho_{13}$, $\nu_{12} = \nu_{13}$, $\sigma_{12} = - \sigma_{13}$, 
$\sigma_{23} = 0$. 

\begin{figure}
\centerline{\includegraphics[width=16cm]{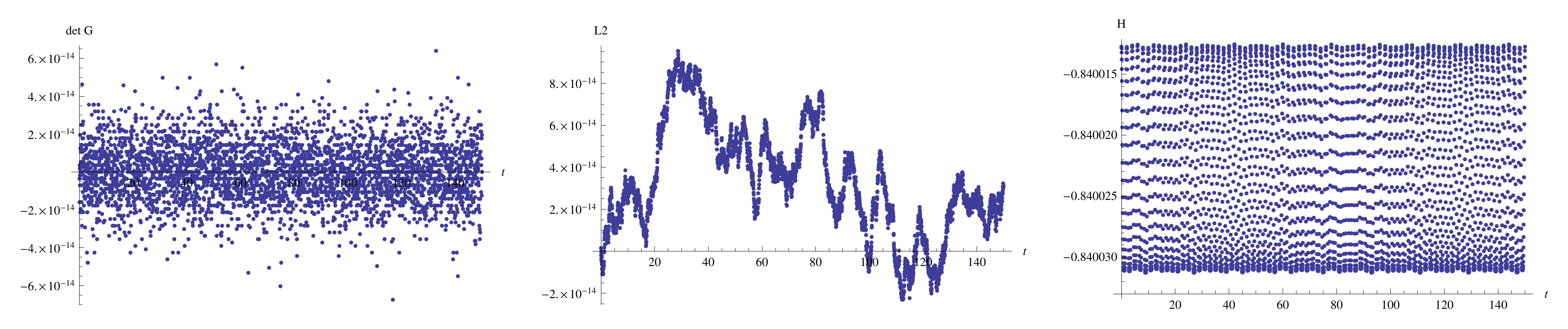}}
\caption{Evolution of $\det G$, $|| \mathbf{L}_c||^2$, and $H$ for 25 rounds of the figure 8 with 150 steps each of stepsize $h=0.04$.
The typical behaviour of Poisson integrators is found that preserves the Casimirs essentially to machine precession
while the Hamiltonian has larger fluctuations but no drift.}
\end{figure}

The period $18$ orbit of $\Phi^h$ is given by $\rho_{13} = 2.33107$, $\nu_{13} =  2.35105$, and $\sigma_{13} = 1.28227$
for $h = 1/3$, $T=6$. 
The period $6 m$ orbit of $\Phi^h$ with stepsize $h=1/m$ converges to a phase space point on the figure 8 orbit 
of the continuous system with period $T=6$ and coordinates $\rho_{13} = 2.34791$, $\nu_{13} =2.3746$, $\sigma_{13} = 1.28904$. 
The corresponding value of the Hamiltonian is $H = -0.84$. 
This orbit for $m= 25$ is shown in Figure~1. 
Note that in full space this corresponds to only half of the figure 8.
In order to illustrate the properties of the Poisson integrator 
Figure~2 shows the Casimirs and the Energy over 25 rounds of the figure 8.

Instead of scaling the stepsize with $m$ we can also fix it at say $h = 1$ 
and scale the initial conditions with $m$.
Then we obtain the statement that the map $\Phi^1$ has a family of periodic orbits of period $6m$ 
for which $\rho_{13} \to 2.34791 m^{4/3}$, $\nu_{13} \to 2.3746 m^{-2/3}$, $\sigma_{13} \to 1.28904 m^{1/3}$
for large $m$.

\nfootnote{To be honest, the simple splitting integrator in the original coordinates (but in 2D!) is actually faster by a factor of two
(for the same stepsize) even though its dimension is larger.}

\section*{Acknowledgement}

I would like to thank Anthony Henderson and Konrad Sch\"obel  for useful discussions.
This research was supported in part by ARC Discovery Grant DP110102001.
This paper was completed during a stay at the Fields Institute and the author would 
like to thank the Institute for its hospitality.


\begin{thebibliography}{10}

\bibitem{Albouy04}
Alain Albouy.
\newblock Mutual distances in clestial mechanics, {L}ectures at {N}ankai
  institute, {T}ianjin, {C}hina.
\newblock Technical report, IMCCE Paris, France, 2004.

\bibitem{AlbouyChenciner98}
Alain Albouy and Alain Chenciner.
\newblock Le probl\`eme des {$n$} corps et les distances mutuelles.
\newblock {\em Invent. Math.}, 131(1):151--184, 1998.

\bibitem{BolsinovBorisovMamaev99}
A.~V. Bolsinov, A.~V. Borisov, and I.~S. Mamaev.
\newblock Lie algebras in vortex dynamics and celestial mechanics. {IV}.
\newblock {\em Regul. Chaotic Dyn.}, 4(1):23--50, 1999.

\bibitem{daSilva01}
Ana Cannas~da Silva.
\newblock {\em Lectures on symplectic geometry}, volume 1764 of {\em Lecture
  Notes in Mathematics}.
\newblock Springer-Verlag, Berlin, 2001.

\bibitem{CM00}
Alain Chenciner and Richard Montgomery.
\newblock A remarkable periodic solution of the three-body problem in the case
  of equal masses.
\newblock {\em Ann. of Math. (2)}, 152(3):881--901, 2000.

\bibitem{Egilsson95}
{\'A}g{\'u}st~Sverrir Egilsson.
\newblock On embedding the {1:1:2} resonance space in a {P}oisson manifold.
\newblock {\em Electron. Res. Announc. Amer. Math. Soc.}, 1(2):48--56
  (electronic), 1995.

\bibitem{GSS88}
Martin Golubitsky, Ian Stewart, and David~G. Schaeffer.
\newblock {\em Singularities and groups in bifurcation theory. {V}ol. {II}},
  volume~69 of {\em Applied Mathematical Sciences}.
\newblock Springer-Verlag, New York, 1988.

\bibitem{HaLuWa02}
E.~Hairer, C.~Lubich, and G.~Wanner.
\newblock {\em Geometric numerical integration}.
\newblock Springer, Berlin, 2002.

\bibitem{Harris92}
Joe Harris.
\newblock {\em Algebraic geometry}, volume 133 of {\em Graduate Texts in
  Mathematics}.
\newblock Springer-Verlag, New York, 1992.
\newblock A first course.

\bibitem{Lagrange1772}
J.L. Lagrange.
\newblock Essai sur le probl{\'e}me des trois corps.
\newblock In {\em {\OE}uvres comple{\'e}tes}, volume~6, pages 229--331. 1772.

\bibitem{LeimkuhlerReich04}
B.~Leimkuhler and S.~Reich.
\newblock {\em Simulating {H}amiltonian dynamics}.
\newblock Cambridge University Press, 2004.

\bibitem{LMS93}
Eugene Lerman, Richard Montgomery, and Reyer Sjamaar.
\newblock Examples of singular reduction.
\newblock In {\em Symplectic geometry}, volume 192 of {\em London Math. Soc.
  Lecture Note Ser.}, pages 127--155. Cambridge Univ. Press, Cambridge, 1993.

\bibitem{LittlejohnReinsch97}
Robert~G. Littlejohn and Matthias Reinsch.
\newblock Gauge fields in the separation of rotations and internal motions in
  the n-body problem.
\newblock {\em Rev. Mod. Phys.}, 69(1):213--275, 1997.

\bibitem{MarsdenRat94}
J.~E. Marsden and T.~S. Ratiu.
\newblock {\em Introduction to Mechanics and Symmetry}.
\newblock Springer, New York, 1994.

\bibitem{MW74}
J.~E. Marsden and A.~Weinstein.
\newblock Reduction of symplectic manifolds with symmetry.
\newblock {\em Rep. on Math. Phys.}, 5:121--130, 1974.

\bibitem{McLaQu02}
R.~I. McLachlan and G.~R. Quispel.
\newblock Splitting methods.
\newblock {\em Acta Numer.}, 11:341--434, 2002.

\bibitem{Schwarz75}
Gerald~W. Schwarz.
\newblock Smooth functions invariant under the action of a compact {L}ie group.
\newblock {\em Topology}, 14:63--68, 1975.

\bibitem{SudarshanMukunda74}
E.~C.~G. Sudarshan and N.~Mukunda.
\newblock {\em Classical dynamics: a modern perspective}.
\newblock Wiley-Interscience [John Wiley \& Sons], New York, 1974.

\bibitem{Whittaker37}
E.~T. Whittaker.
\newblock {\em A Treatise on the Analytical Dynamics of Particles and Rigid
  Bodies}.
\newblock Cambridge University Press, Cambridge, 4 edition, 1937.

\bibitem{Yoshida90}
Haruo Yoshida.
\newblock Construction of higher order symplectic integrators.
\newblock {\em Phys. Lett. A}, 150(5-7):262--268, 1990.

\end{thebibliography}

\def\cprime{$'$}

\end{document}